\documentclass[twocolumn]{ieeeconf} % Comment this line out if you need a4paper
\usepackage{amsmath,amsfonts}
\usepackage{graphicx}
\usepackage{color}

\newtheorem{lemma}{Lemma}
\newtheorem{remark}{Remark}
\newtheorem{theorem}{Theorem}
\newtheorem{assum}{Assumption}
\newtheorem{proposition}{Proposition}
\newtheorem{definition}{Definition}

\newcommand{\R}{\mathbb{R}}
\newcommand{\K}{\mathcal{K}}
\newcommand{\KL}{\mathcal{KL}}

\newcommand{\D}{\mathcal{D}}

\IEEEoverridecommandlockouts                              % This command is only needed if 
\title{\LARGE \bf
  Robust output regulation of $2 \times 2$ hyperbolic systems part I: Control law and Input-to-State Stability\\
}

\author{Pierre-Olivier Lamare$^{1}$, Jean Auriol$^{2}$, Florent Di Meglio$^{3}$, and Ulf Jakob F. Aarsnes$^4$% <-this % stops a space
  \thanks{
    $^{1}$Pierre-Olivier Lamare is with Universit\'e Nice C\^ote d'Azur, Inria BIOCORE, BP93, 06902 Sophia-Antipolis Cedex, France.
    {\tt\small pierre-olivier.lamare@inria.fr}}%
  \thanks{
    $^{2}$Jean Auriol is with MINES ParisTech, PSL Research University, CAS - Centre automatique et syst\`emes, 60 bd St Michel 75006 Paris, France. 
    {\tt\small jean.auriol@mines-paristech.fr}}%
  \thanks{$^{3}$Florent Di Meglio is with MINES ParisTech, PSL Research University, CAS - Centre automatique et syst\`emes, 60 bd St Michel 75006 Paris, France. 
    {\tt\small florent.di\_meglio@mines-paristech.fr}}%
  \thanks{$^{4}$Ulf Jakob F. Aarsnes is International Research Institute of Stavanger (IRIS), Oslo, Norway and DrillWell - Drilling and well centre for improved recovery, Stavanger, Norway
    {\tt\small ujfa@iris.no}}
}

\begin{document}

\maketitle
\thispagestyle{empty}
\pagestyle{empty}

%%%%%%%%%%%%%%%%%%%%%%%%%%%%%%%%%%%%%%%%%%%%%%%%%%%%%%%%%%%%%%%%%%%%%%%%%%%%%%%% 
\begin{abstract}
  We consider the problem of output feedback regulation for a linear first-order hyperbolic system with collocated input and output in presence of a general class of disturbances and noise. The proposed control law is designed through a backstepping approach incorporating an integral action. To ensure robustness to delays, the controller only cancels part of the boundary reflection by means of a tunable parameter. This also enables a trade-off between disturbance and noise sensitivity. We show that the boundary condition of the obtained target system can be transformed into a Neutral Differential Equation (NDE) and that this latter system is Input-to-State Stable (ISS). This proves the boundedness of the controlled output for the target system. This extends previous works considering an integral action for this kind of system~\cite{LDM16}, and constitutes an important step towards practical implementation of such controllers. Applications and practical considerations, in particular regarding the system's sensitivity functions are derived in a companion paper.
\end{abstract}

\section{Introduction}
\label{sec:introduction}

In this paper, we solve the problem of output feedback regulation for a system composed of two linear hyperbolic PDEs with collocated boundary input and output in presence of disturbances and noise in the measurements. The proposed controller combines a backstepping approach with an integral action. The resulting feedback law is proved to be Input-to-State Stable (ISS). This paper extends the results stated in~\cite{LDM16} where uncorrupted anti-collocated measurements were considered in presence of static disturbances.

A large number of physical networks may be represented by hyperbolic systems. Among them we can cite the hydraulic networks~\cite{Bastin2011,DSBCAN08}, road traffic networks~\cite{FHS14}, oil well drilling~\cite{A13,DMBPA14} or gas pipeline networks~\cite{GDL11}. Due to the importance of such applications, a large number of results concerning their control has emerged this last decade. Among the different challenges, the disturbance rejection problem has been recently considered in~\cite{A13,AA15,D16,D17,DSBCAN08,LBL15,TK14}. In~\cite{A13,AA15}, the rejection of a perturbation affecting the uncontrolled boundary side of a $2 \times 2$ linear hyperbolic system is solved using a backstepping approach. In~\cite{LBL15}, a proportional-integral controller is introduced to ensure the stabilization of a reference trajectory. An integral action is considered in~\cite{DSBCAN08} to ensure output rejection and its effectiveness is validated on experimental data. In~\cite{TK14}, a sliding mode control approach is used to reject a boundary time-varying input disturbance. % The results related to the present paper are the results in~\cite{LBL15} for the integral control,~\cite{DMVK13,VKC11} for the control and observer by backstepping.

The main contribution of this paper is to solve the problem of output disturbance rejection for a $2 \times 2$ first-order hyperbolic system with collocated boundary input and output. % This is done by incorporating an integral action to a backstepping-based controller. It shall be point out that it is much more difficult to deal with the collocated case than with the anti-collocated (cf.~\cite{LDM16}).
Besides, the class of disturbances considered in this paper, namely bounded signals, is more general than the one proposed in~\cite{D16,D17} in which the disturbance signal is generated by an exosystem of finite dimension, or than the smooth disturbances considered in~\cite{LBL15,LDM16}. 
%They are assumed to vary in time and the measured output is subject to noise. Moreover, we assume very large classes of disturbances, namely bounded signals. Thus, it is more general than the disturbances considered in~\cite{D16,D17} in which the disturbance signal is generated by an exosystem of finite dimension, or the regular disturbances considered in~\cite{LBL15,LDM16}.

Our approach is the following. Similarly to~\cite{LDM16}, the original system is mapped to a simple target system where an integral term is added. The disturbances are incorporated into the target system. To state that the resulting target system is ISS with respect to perturbations and noise, we show that the output satisfies a Neutral Differential Equation (NDE). Using existing results on such systems, the ISS property is finally obtained.

The paper is organized as follows. The original disturbed system and the notations are introduced in Section~\ref{sec:pb_desciption}. In Section~\ref{sec:output-regulation}, we present the stabilization result: using a backstepping transformation, the original system is mapped to a target system for which the in-domain couplings are removed. The control law is then designed. The resulting closed-loop system can be rewritten as a neutral delay-equation which is proved to be ISS with respect to the noise and the disturbances. To envision practical application, an observer-controller is introduced in Section~\ref{sec:boundary-observer}. In Section~\ref{sec:feedb-outp-regul} we prove that the resulting output feedback control law still stabilizes the output. Besides, it is shown that static disturbances are completely rejected. This result has already been proved in~\cite{LDM16} in the case of an uncorrupted measurement and for anti-collocated input and output.

%%% Local Variables:
%%% mode: latex
%%% TeX-master: "ACC_I_Theoretical"
%%% End:

\section{Problem Description}
\label{sec:pb_desciption}

We consider the following system
\begin{align}
  \label{eq:perturbed_system_1} u_t(t,x) + \lambda(x)
  u_x(t,x) & = \gamma_1(x) v(t,x) + d_1(t)m_1(x) \\
  \label{eq:perturbed_system_2}
  v_t(t,x) -\mu(x) v_x(t,x) & = \gamma_2(x) u(t,x) + d_2(t)m_2(x) \, ,
	\end{align}
	under the boundary conditions
	\begin{align}
  \label{eq:perturbed_system_3} u(t,0) & = q v(t,0)
                                         + d_3(t) \\
  \label{eq:perturbed_system_4}
  v(t,1) & = \rho u(t,1) + U(t) + d_4(t) \, ,
\end{align}
where~$t \in \left[0,+\infty\right)$ is the time variable,~$x \in
\left[0,1\right]$ is the space variable,~$q\neq0$ is a constant parameter, and~$U$ is the control
input. The initial conditions~$u^0(x)=u(0,x)$ and~$v^0(x)=v(0,x)$ are assumed to be bounded and therefore in~$L^\infty((0,1);\R)$. We make the following assumption on the velocities~$\lambda$ and~$\mu$ and on the in-domain-coupling terms~$\gamma_1$ and~$\gamma_2$.
\begin{assum}
  \label{assum:function_regularity}
  The functions~$\lambda$,~$\mu:[0,1]\rightarrow \R$ are Lipschitz-continuous and satisfy~$\lambda(x)$,~$\mu(x)>0$, for all~${x\in[0,1]}$. The functions~$\gamma_1$,~$\gamma_2$ belong to~$C^1([0,1];\R)$. The product of the distal reflection $q$ with the proximal reflection $\rho$ is assumed to be strictly lower than one to ensure delay-robustness \cite{A17}.
\end{assum}
The functions~$d_1$ and $d_2$ correspond to disturbances acting on the
right-hand side of~\eqref{eq:perturbed_system_1} and
\eqref{eq:perturbed_system_2}. The locations of these distributed
disturbances are given by the unknown functions~$m_1$ and~$m_2$. The
functions~$d_3$ and~$d_4$ correspond to disturbances acting on the
right-hand side of~(\ref{eq:perturbed_system_3}) and (\ref{eq:perturbed_system_4}), respectively.  

Moreover, we assume that the measured output is also subject to an unknown noise~$n(t)$
\begin{equation}
  y_m(t) = u(t,1) + n(t) \, .  
\end{equation}
The aim of this paper is to regulate the output
\begin{equation}
  y(t) = u(t,1) \, .
\end{equation}
Let state the following assumption on the disturbances.
\begin{assum}
  \label{assum:disturbance_regularity}
  The disturbances~$d_i$,~$i=1,\dots,4$, are in~$W^{2,\infty}\left((0,\infty);\R\right)$, the noise~$n$ is assumed to be in~$L^\infty((0,\infty);\R)$, and the disturbance input locations~$m_1$ and~$m_2$ are in~$C\left([0,1];\R^+\right)$.
\end{assum}
With the two former assumptions, using the characteristics method and classical fixed point arguments we have the following result (see e.g.~\cite{B00}).
\begin{theorem}
  \label{theo:wellposedness}
  Under Assumptions~\ref{assum:function_regularity} and~\ref{assum:disturbance_regularity} system~\eqref{eq:perturbed_system_1}--\eqref{eq:perturbed_system_4} admits an unique solution in ~$C\left(\left[0,\infty\right);L^\infty\left((0,1);\R^2\right) \cap  L^1\left((0,1);\R^2\right)\right)$.
\end{theorem}
We denote by~$E'$ the set of bounded functions~$y : [0,1] \rightarrow \R^2$. Therefore,~$E'$ belongs to~$L^\infty((0,1);\R^2)$ and let $E:=E'\times \R$. The notation~$\left\lVert y \right\rVert_{E'}$ refers to~$\left\lVert y \right\rVert_{L^\infty((0,1);\R^2)}$ and for $z=\left(z_1,z_2,z_3\right)\in E' \times \R$, $\left\lVert z \right\rVert_E = \left\lVert \left(z_1,z_2\right) \right\rVert_{E'} + \left|z_3\right|$.

%%% Local Variables:
%%% mode: latex
%%% TeX-master: "ACC_I_Theoretical"
%%% End:

\section{Output Regulation}
\label{sec:output-regulation}

To achieve output regulation we choose to design a controller combining a backstepping controller $U_{BS}$ and an integrator term $k_I \eta$, namely
\begin{align}
  \label{eq:Controller}
  U(t) & = U_{BS}(t) +  k_I \eta(t) \\
  \label{eq:eta_dot}
  \dot{\eta} (t) & = y_m(t) \, .  
\end{align}
In what follows, we design~$U_{BS}$ and~$k_I$ to perform output
regulation. We make the assumption of full-state measurement. In the next section, using a backstepping transformation, we map the original system~\mbox{\eqref{eq:perturbed_system_1}--\eqref{eq:perturbed_system_4}} to a simple target system from which the in-domain couplings have been removed. 

\subsection{Backstepping Transformation and Target System}
Let us consider  the backstepping transformation $\Gamma_1[(u,v)(t)](\cdot) =
\alpha(t,\cdot)$ and ${\Gamma_2[(u,v)(t)](\cdot) = \beta(t,\cdot)}$ defined by
\begin{align}
   \alpha(t,x) & = u(t,x) - \int_0^x  K^{uu}(x,\xi)u(t,\xi) d\xi \nonumber \\
               & \hphantom{=} - \int_0^xK^{uv}(x,\xi)v(t,\xi)d\xi \label{transfo_11}\\
  \beta(t,x) & = v(t,x) - \int_0^x K^{vu}(x,\xi)u(t,\xi)d\xi \nonumber \\
               & \hphantom{=} -\int_0^xK^{vv}(x,\xi)v(t,\xi)d\xi \, , \label{transfo_12}
               % \gamma(t) & = \left(1+a\right)\eta(t) - \int_0^1 \left(\ell_1(\xi)u(t,\xi) +\ell_2(\xi)v(t,\xi)\right) d\xi \, .
\end{align}
where the kernels~$K^{uu}, K^{uv}, K^{vu}$, and~$K^{vv}$ are defined in~\cite{CVKB13} in~$L^\infty(\mathcal{T})$, where~$\mathcal{T}=\{(x,\xi)\in [0,1]^2 |\quad \xi \leq x\}$. %by the following set of hyperbolic PDEs:
%\begin{align}
%\lambda(x)K^{uu}_x+\lambda(\xi)K^{uu}_{\xi} &=-\gamma_2(\xi)K^{uv}-\lambda'(\xi)K^{uv} \label{kerneK^{uu}}\\
%  \lambda(x)K^{uv}_x-\mu(\xi)K^{uv}_{\xi} &=-\gamma_1(\xi)K^{uu}+\mu'(\xi)K^{uu} \label{kerneK^{uv}}\\
%  \mu(x)K^{vu}_x-\lambda(\xi)K^{vu}_{\xi} & =\gamma_2(\xi)K^{vv}+\lambda'(\xi)K^{vv} \label{kerneK^{vu}}\\
%  \mu(x)K^{vv}_x+\mu(\xi)K^{vv}_{\xi} & = \gamma_1(\xi)K^{vu}-\mu'(\xi)K^{vu} \, , \label{kerneK^{vv}}
%\end{align}
%with the following set of boundary conditions:
%\begin{align}
%  K^{vu}(x,x) & =-\frac{\gamma_2(x)}{\lambda(x)+\mu(x)} \\
%  K^{vv}(x,0) & =\frac{\lambda(0)q}{\mu(0)}K^{vu}(x,0)\\
%  K^{uv}(x,x) & =\frac{\gamma_1(x)}{\lambda(x)+\mu(x)} \\
%  K^{uv}(x,0) & =\frac{\lambda(0)q}{\mu(0)}K^{uu}(x,0) \, . \label{K_bond}
%\end{align}
We recall the following lemma
\begin{lemma}[~\cite{CVKB13}]
 The transformation~\eqref{transfo_11}--\eqref{transfo_12} is invertible and the inverse transformation can be expressed as follow
\begin{align}
  u(t,x)& = \alpha(t,x) + \int^x_0 L^{\alpha\alpha}(x,\xi)\alpha(t,\xi)d\xi \nonumber \\
        & \hphantom{=} + \int^x_0L^{\alpha\beta}(x,\xi)\beta(t,\xi)d\xi\label{back_inv1} \\
  v(t,x)& =\beta(t,x)  + \int^x_0 L^{\beta\alpha}(x,\xi)\alpha(t,\xi)d\xi \nonumber \\
        & \hphantom{=} + \int^x_0 L^{\beta\beta}(x,\xi)\beta(t,\xi)d\xi \, , \label{back_inv2}
\end{align}
where~$L^{\alpha\alpha}$,~$L^{\alpha\beta}$,~$L^{\beta\alpha}$, and~$L^{\beta\beta}$ belong to~$L^{\infty}(\mathcal{T})$.
\end{lemma}
The transformation \eqref{transfo_11}-\eqref{transfo_12} maps the original system~\mbox{\eqref{eq:perturbed_system_1}--\eqref{eq:perturbed_system_4}} to the following target system
\begin{align}
  & \alpha_t + \lambda(x)\alpha_x = \D_1(t)M_1(x) \label{target_1} \\
  & \beta_t - \mu(x)\beta_x = \D_2(t)M_2(x) \, ,
\end{align}
with the boundary conditions
\begin{align}
  \alpha(t,0) & = q \beta(t,0) + d_3(t) \\
  \beta(t,1) & = \rho\int^1\left(L^{\alpha\alpha}(1,\xi)\alpha(t,\xi)
               + L^{\alpha\beta}(1,\xi)\beta(t,\xi)\right)d\xi \nonumber \\
              & \hphantom{=} -  \int_0^1 \left(L^{\beta\alpha}(1,\xi)\alpha(t,\xi)+L^{\beta\beta}(1,\xi)\beta(t,\xi)\right)d\xi \nonumber \\
  & \hphantom{=} +\rho \alpha(t,1) + U(t)+ k_I \eta(t) + d_4(t) \, , \label{target_BC}
\end{align}
where
\begin{align}
\dot{\eta}(t)& = \alpha(t,1)+n(t) \nonumber \\
             & \hphantom{=} + \int^1_0(L^{\alpha\alpha}(1,\xi)\alpha(t,\xi)+L^{\alpha\beta}(1,\xi)\beta(t,\xi))d\xi \, , \label{eq_eta}
\end{align}
with
\begin{align}
  \D_1(t)M_1(x) & = d_1(t)m_1(x)-K^{uu}(x,0)\lambda(0)d_3(t)  \nonumber \\
                & \hphantom{=} -\int_0^x K^{uu}(x,\xi)d_1(t)m_1(\xi)d\xi \nonumber \\
                & \hphantom{=} - \int_0^x K^{uv}(x,\xi)d_2(t)m_2(\xi) d\xi
\end{align}
\begin{align}
\D_2(t)M_2(x)& = d_2(t)m_2(x)-K^{vu}(x,0)\lambda(0)d_3(t) \nonumber \\
                & \hphantom{=} -\int_0^x K^{vu}(x,\xi)d_1(t)m_1(\xi)d\xi \nonumber \\
                & \hphantom{=} -\int_0^x K^{vv}(x,\xi)d_2(t)m_2(\xi)d\xi \, .
\end{align}
Note that if~$\dot{\eta}$ converges to zero and if~$n(t)=0$, then~$u(t,1)$ converges to 0, due to \eqref{back_inv1}. Unconventionally, we define the control law $U_{BS}$ in terms of the variables of the target system~$\alpha$ and $\beta$ as
\begin{align}
  U_{BS}(t)& = -\tilde{\rho} \alpha(t,1) \nonumber \\
	& \hphantom{=}-\rho \int^1_0\left(L^{\alpha\alpha}(1,\xi)\alpha(t,\xi)+L^{\alpha\beta}(1,\xi)\beta(t,\xi)\right)d\xi \nonumber \\
      & \hphantom{=} +  \int_0^1 \left(L^{\beta\alpha}(1,\xi)\alpha(t,\xi)+L^{\beta\beta}(1,\xi)\beta(t,\xi)\right)d\xi  \nonumber  \\
      &\hphantom{=} -k_I\int_0^1 \left(l_1(\xi) \alpha(t,\xi) + l_2(\xi) \beta(t,\xi)\right)d\xi, \label{eq:U_2}
\end{align}
where the tuning parameter $\tilde{\rho}$ satisfies 
\begin{align}
|\rho q|+|\tilde{\rho} q|<1,
\end{align}
which is well defined since $\rho q<1$. The functions~$l_1$ and~$l_2$ on the interval~$[0,1]$ are defined as the solution of the system
\begin{align}
  & (l_1(x)\lambda(x))'=L^{\alpha\alpha}(1,x) \label{eq_l1}\\
  & (l_2(x)\mu(x))'=-L^{\alpha\beta}(1,x) \,  \label{eq_l2},
\end{align}
with the boundary conditions
\begin{align}
  l_2(1)=0 \, , \qquad l_1(0)=\frac{\mu(0)}{q\lambda(0)}l_2(0). \label{bound_l}
\end{align}
This control law is composed of two parts that have two distinct effects. The first one (made of the three first lines) corresponds to the control law derived in~\cite{A17}. It would stabilize the original system in the absence of disturbances and of the integral term $k_i\eta(t)$. Note that the purpose of the term $-\tilde{\rho}\alpha(t,1)$ is to avoid a complete cancellation of the proximal reflexion and thus to guarantee some delay-robustness~\cite{A17}. The second term of the control law (made of the last line of~\eqref{eq:U_2}) is related to the integral action. In order to ensure the existence of a solution to~\eqref{eq_l1}-\eqref{bound_l}, we make the following assumption
\begin{assum}
  \label{assum:condition_neq_L}
\begin{align}
1+\int_0^1 L^{\alpha \alpha}(1,\xi)d\xi +\frac{1}{q}\int_0^1 L^{\alpha \beta}(1,\xi)d\xi \ne 0\label{funda_eq} \, .
\end{align}
\end{assum}
Unfortunately, this assumption has no physical interpretation.
Using equation~\eqref{back_inv1}-\eqref{back_inv2}, one can write the control law~\eqref{eq:U_2} in terms of the original variables $u$ and $v$.
%Using~\eqref{transfo_11},~\eqref{transfo_12},~\eqref{back_inv1}, and~\eqref{back_inv2} the control law in~\eqref{eq:U_2} %may be written in the state coordinate~$u$ and~$v$ as
%\begin{align}
 % U(t) & = -\tilde{\rho}u(t,1) - \left(\rho-\tilde{\rho}\right)\int_0^1K^{uu}(1,\xi)u(t,\xi)d\xi \nonumber \\
 %      & \hphantom{=} - \left(\rho-\tilde{\rho}\right)\int_0^1K^{uv}(1,\xi)v(t,\xi)d\xi \nonumber \\
 %      & \hphantom{=} + \int_0^1 K^{vu}(1,\xi)u(t,\xi)d\xi+ \int_0^1 K^{vv}(1,\xi)v(t,\xi)d\xi \nonumber \\
 %      & \hphantom{=} -k_I\int_0^1 l_1(\xi)\Gamma_1[(u,v)(t)](\xi)d\xi \nonumber \\
 %      & \hphantom{=} - k_I \int_0^1 l_2(\xi)\Gamma_2[(u,v)(t)](\xi)d\xi  \, . \label{eq:U}
%\end{align}
In the next sections, we prove that this control law ensures output regulation. We first investigate a pseudo-steady state of the closed loop system. 

\subsection{Pseudo-steady state}

In this section, we consider a pseudo-steady state of the target system~\eqref{target_1}--\eqref{eq_eta} in presence of the control law~\eqref{eq:U_2}, that corresponds to~$u^{ss}(t,1)=\alpha(t,1)+\int_0^1 L^{\alpha \alpha}(1,\xi) \alpha^{ss}(t,\xi)d\xi +\int_0^1 L^{\alpha \beta}(1,\xi) \beta^{ss}(t,\xi) d\xi=0$. We then derive the error system, i.e the difference between the real state and this pseudo-steady state. This pseudo steady-state is defined by
\begin{align}
  \frac{d}{dx}\begin{pmatrix} \alpha^{ss}(t,x) \\ \beta^{ss}(t,x) \end{pmatrix}= \begin{pmatrix} \frac{\D_1(t)M_1(x)}{\lambda(x)} \\ -\frac{\D_2(t)M_2(x)}{\mu(x)} \end{pmatrix} \label{steady_eq}
\end{align}
along with the initial conditions
\begin{align}
  \beta^{ss}(t,0) & =\frac{1}{q}(\alpha^{ss}(t,0)-d_3(t))  \label{steady_boun1}\\
  \alpha^{ss}(t,1) & =-\int_0^1 L^{\alpha \alpha}(1,\xi) \alpha^{ss}(t,\xi)d\xi \nonumber \\
  & \hphantom{=} -\int_0^1 L^{\alpha \beta}(1,\xi) \beta^{ss}(t,\xi) d\xi \, .\label{steady_boun2}
\end{align}
We have the following lemma regarding the existence of a solution to
the ODE~\eqref{steady_eq},~\eqref{steady_boun1}, and~\eqref{steady_boun2}.
\begin{lemma}
  \label{lem:existence_ss}
If equation~\eqref{funda_eq} holds, the ordinary differential
equation~\eqref{steady_eq} with boundary
conditions~\eqref{steady_boun1} and~\eqref{steady_boun2} has a unique
solution. Moreover, for every $x \in [0,1]$ one has
$\alpha^{ss}(\cdot,x)$ and $\beta^{ss}(\cdot,x)$ in~$W^{2,\infty}\left((0,\infty);\R\right)$.
\end{lemma}
\begin{proof}
Let us define the matrix~$A_1$ by 
\begin{align}
A_1=\begin{pmatrix}
	1+\int_0^1 L^{\alpha \alpha}(1,\xi)d\xi & \int_0^1 L^{\alpha \beta}(1,\xi)d\xi \\ -\frac{1}{q} & 1
\end{pmatrix}.
\end{align}
Due to~\eqref{funda_eq}, this matrix is invertible. We then define~${a = \begin{pmatrix}
	a_1 & a_2
\end{pmatrix}^\top}$ by $a  = A_1^{-1} b$
%\begin{align}
% \, ,\
%\end{align}
with~$b = \begin{pmatrix}b_1 & b_2\end{pmatrix}^\top$ where
\begin{align}
  b_1 & = \int_0^1 L^{\alpha \alpha}(1,\xi) \int_\xi^1 \frac{\D_1(t)M_1(\nu)}{\lambda(\nu)}d\nu d\xi \nonumber \\
      & \hphantom{=} + \int_0^1L^{\beta \alpha}(1,\xi) \int_0^\xi \frac{\D_2(t)M_2(\nu)}{\mu(\nu)}d\nu d\xi \\
  b_2 & = -\frac{d_3(t)}{q}-
        \int_0^1\frac{\D_1(t)M_1(\xi)}{q\lambda(\xi)}d\xi \, .
\end{align}
One can thencheck that the function
\begin{align}
\begin{pmatrix}\alpha^{ss}(t,x)  \\ \beta^{ss}(t,x) \end{pmatrix}=\begin{pmatrix}
	a_1-\int_x^1 \frac{\D_1(t)M_1(\xi)}{\lambda(\xi)}d\xi \\ a_2 -\int_0^x \frac{\D_2(t)M_2(\xi)}{\mu(\xi)}d\xi
\end{pmatrix},
\end{align}
is solution of~\eqref{steady_eq} with the boundary
conditions~\eqref{steady_boun1} and \eqref{steady_boun2}. This concludes the proof of Lemma~\ref{lem:existence_ss}.
\end{proof}
Let us state
\begin{align}
\label{eq:eta_ss}
\eta^{ss}(t)& = \frac{\beta^{ss}(t,1)-(\rho-\tilde{\rho})\alpha^{ss}(t,1)-d_4(t)}{k_I} \nonumber \\
  & \hphantom{=} + \int_0^1 \left(l_1(\xi) \alpha^{ss}(t,\xi) + l_2(\xi) \beta^{ss}(t,\xi)\right)d\xi \, .
\end{align}
By defining the error variables $\bar{\alpha} =\alpha -\alpha^{ss}$, $\bar{\beta}=\beta -\beta^{ss}$, and $\bar{\eta}=\eta -\eta^{ss}$, one gets the following system
\begin{align}
\bar{\alpha}_t + \lambda(x) \bar{\alpha}_x=-\alpha_t^{ss} \label{eq_alpha_bar} \\
\bar{\beta}_t -\mu(x) \bar{\beta}_x=-\beta_t^{ss} \, , \label{eq_beta_bar}
\end{align}
with the boundary conditions 
\begin{align}
\bar{\alpha}(t,0)&=q\bar{\beta}(t,0) \label{bound_alpha_bar}\\
\bar{\beta}(t,1) &=\left(\rho-\tilde{\rho}\right)\bar{\alpha}(t,1)+k_I\bar{\eta}(t)\nonumber \\
  & \hphantom{=} -k_I\int_0^1 \left(l_1(\xi) \bar{\alpha}(t,\xi)+l_2(\xi) \bar{\beta}(t,\xi)\right)d\xi \, .
\end{align}
Noticing that~$\alpha^{ss}(t,1)=-\int_0^1 L^{\alpha\alpha}(1,\xi)\alpha^{ss}(t,\xi)d\xi-\int_0^xL^{\alpha\beta}(1,\xi)\beta^{ss}(t,\xi)d\xi$, we also have that 
\begin{align}
\dot{\bar{\eta}}(t) & =\int^1_0 \left(L^{\alpha\alpha}(1,\xi)\bar{\alpha}(t,\xi)+L^{\alpha\beta}(1,\xi)\bar{\beta}(t,\xi)\right)d\xi \nonumber \\
  & \hphantom{=}+ \bar{\alpha}(t,1)+n(t) -\dot{\eta}^{ss}(t) \, \label{eta_bar_eq}.
\end{align}

\subsection{Stability Analysis}
In this section, we analyze the stability properties of system~\eqref{eq_alpha_bar}--\eqref{eta_bar_eq}. More precisely, we derive conditions on $k_I$ that ensure the Input-to-State Stability of system~\eqref{eq_alpha_bar}--\eqref{eta_bar_eq}. The proof will be done in three steps. First, using a simple transformation, we rewrite the system~\eqref{eq_alpha_bar}--\eqref{eta_bar_eq} as a neutral-delay equation (NDE). We then recall some conditions that guarantee the stability of this NDE in the absence of disturbances. Finally, we prove that these conditions imply the Input-to-State Stability.  
Let us consider the inversible transformation
\begin{align}
\gamma(t)=\bar{\eta}(t)-\int_0^1 \left( l_1(\xi) \bar{\alpha}(t,\xi)+l_2(\xi) \bar{\beta}(t,\xi)\right) d\xi.
\end{align}
System~\eqref{eq_alpha_bar}--\eqref{eta_bar_eq} rewrites 
\begin{align}
&\bar{\beta}(t,1)=(\rho-\tilde{\rho})\bar{\alpha}(t,1)+k_I\gamma(t) \label{bound_beta_bar}\\
&\dot{\gamma}(t)=\left(1+l_1(1)\lambda(1)\right) \bar{\alpha}(t,1) +n(t) -\dot{\eta}^{ss}(t). \label{bound_gamma}
\end{align}
Using~\eqref{eq_l1} and~\eqref{bound_l}, we have
\begin{align}
1+l_1(1)\lambda(1)% &=1+l_1(0)\lambda(0)+\int_0^1 L^{\alpha \alpha}(1,\xi) d\xi \nonumber \\
                  &=1+l_2(0)\frac{\mu(0)}{q}+\int_0^1 L^{\alpha \alpha}(1,\xi) d\xi \nonumber \\
                  &=1+\frac{1}{q}\int_0^1 L^{\alpha \beta}(1,\xi) d\xi \nonumber \\
                  & \hphantom{=} +\int_0^1 L^{\alpha \alpha}(1,\xi) d\xi.
\end{align}
Thus, due to Assumption~\ref{assum:condition_neq_L},~$1+l_1(1) \lambda(1) \ne 0$. %System~\eqref{eq_alpha_bar},~\eqref{eq_beta_bar} along with the boundary conditions~\mbox{\eqref{bound_alpha_bar}--\eqref{bound_beta_bar}} can be rewritten as
%\begin{align}
%  \label{eq:NDE_1}
%  \bar{\alpha}(t,1) & =(\rho-\tilde{\rho})q\bar{\alpha}\left(t-\int_0^1 \left(\frac{1}{\lambda(\xi)}+\frac{1}{\mu(\xi)}\right) d\xi,1\right) \nonumber \\
%                    & \hphantom{=} +k_Iq\gamma\left(t-\int_0^1 \left(\frac{1}{\lambda(\xi)}+\frac{1}{\mu(\xi)} \right)d\xi\right) \, .
%\end{align}
In the sequel we denote by~$\phi_1(x)$ and~$\phi_2(x)$ the following functions
\begin{equation}
  \phi_1(x) = \int_0^x \frac{1}{\lambda(\xi)}d\xi \, , \quad \phi_2(x) = \int_0^x \frac{1}{\mu(\xi)}d\xi \, ,
\end{equation}
and by~$\tau_1$,~$\tau_2$, and~$\tau$ the following transport times
\begin{align}
  \tau_1  =  \phi_1(1),~\tau_2  = \phi_2(1),~\tau  = \tau_1 + \tau_2 \, . \label{eq_tau}
\end{align}
Using the characteristics method, it is straightforward to show that for all~$t \geq \tau$,
\begin{align}
  \overline{\alpha}(t,1) & = \overline{\alpha}\left(t -\tau_1,0\right) \nonumber \\
  & \hphantom{=} - \int_0^1 \frac{1}{\lambda(\xi)} \alpha_t^{ss} \left(\xi,t - \int_\xi^1 \frac{1}{\lambda(\zeta)}d\zeta\right) d\xi \label{eq:alpha_bar_1} \\
  \overline{\beta}(t,0) & = \overline{\beta}\left(t- \tau_2,1\right) - \int_0^1 \frac{\beta_t^{ss}\left(\xi,t - \phi_2(\xi)\right)}{\mu(\xi)} d\xi \, . \label{eq:beta_bar_0}
\end{align}
Combining these expression with the boundary conditions~\eqref{bound_alpha_bar} and~\eqref{bound_beta_bar}, we get for all~$t\geq \tau$,
\begin{align}
  \overline{\alpha}(t,1) & = q \overline{\beta}\left(t-\tau,1\right) \nonumber \\
  & \hphantom{=} - q \int_0^1 \frac{1}{\mu(\xi)} \beta_t^{ss}\left(\xi, t - \tau_1 -\phi_2(\xi)\right)d\xi \nonumber \\
                         & \hphantom{=} - \int_0^1
                           \frac{1}{\lambda(\xi)} \alpha_t^{ss}
                           \left(\xi,t - \int_\xi^1
                           \frac{1}{\lambda(\zeta)}d\zeta\right) d\xi \label{eq:alpha_1_equality}
\end{align}
Using again boundary condition~\eqref{bound_beta_bar},
relationship~\eqref{eq:alpha_1_equality} becomes
\begin{align}
  \overline{\alpha}(t,1) & = (\rho -\tilde{\rho})q\overline{\alpha}(t-\tau,1)+k_Iq\gamma(t-\tau) \nonumber \\
  & \hphantom{=} - q \int_0^1 \frac{1}{\mu(\xi)} \beta_t^{ss}\left(\xi, t - \tau_1 - \phi_2(\xi)\right)d\xi \nonumber \\
                         & \hphantom{=} - \int_0^1
                           \frac{1}{\lambda(\xi)} \alpha_t^{ss}
                           \left(\xi,t - \int_\xi^1
                           \frac{1}{\lambda(\zeta)}d\zeta\right) d\xi
                           \, . \label{eq:alpha_bar_1_2}
\end{align}
% Thus, it follows
% \begin{align}
%   \label{eq:NDE_perturbed}
%   \dot{\overline{\alpha}}(t,1)  & = \left(\rho - \tilde{\rho}\right)
%   q\dot{\overline{\alpha}}(t-\tau,1)
%   +k_Iq\left(1+l_1(1)\lambda(1)\right)\nonumber \\
%   & \hphantom{=} \times \overline{\alpha}(t-\tau,1) +
%   K(t) \, ,
% \end{align}
% for all~$t \geq \tau$, where
% \begin{align}
%   K(t) & =  (n(t-\tau) - \dot{\eta}^{ss}(t-\tau))k_Iq \nonumber \\
%   & \hphantom{=} - q
%          \int_0^1 \frac{1}{\mu(\xi)} \beta_{tt}^{ss}\left(\xi, t -
%          \tau_1 - \int_0^\xi
%          \frac{1}{\mu(\zeta)}d\zeta\right)d\xi \nonumber \\
%   & \hphantom{=} - \int_0^1
%                            \frac{1}{\lambda(\xi)} \alpha_{tt}^{ss}
%                            \left(\xi,t - \int_\xi^1
%                            \frac{1}{\lambda(\zeta)}d\zeta\right) d\xi
%                          \, .
% \end{align}
% %Two different situations shall be consider. First, let us assume that
%  %~$\tilde{\rho}
%   %\neq \rho$. In this case, 
% 	System~\eqref{eq:NDE_perturbed} is a Neutral Differential Equation (NDE) when $\rho \neq \tilde{\rho}$. 
By differentiating~\eqref{eq:alpha_bar_1_2} with respect to time, one has
\begin{align}
  \label{eq:NDE_perturbed}
  \dot{\overline{\alpha}}(t,1)  & = \left(\rho - \tilde{\rho}\right)
  q\dot{\overline{\alpha}}(t-\tau,1)
  +k_Iq\left(1+l_1(1)\lambda(1)\right) \nonumber \\
  & \hphantom{=} \times \overline{\alpha}(t-\tau,1) +
  K(t) \, ,
\end{align}
where
\begin{align}
  K(t) & =  k_Iq(n(t-\tau) - \dot{\eta}^{ss}(t-\tau)) \nonumber \\
  & \hphantom{=} - q
         \int_0^1 \frac{1}{\mu(\xi)} \beta_{tt}^{ss}\left(\xi, t -
         \tau_1 - \phi_2(\xi)\right)d\xi \nonumber \\
  & \hphantom{=} - \int_0^1
                           \frac{1}{\lambda(\xi)} \alpha_{tt}^{ss}
                           \left(\xi,t - \int_\xi^1
                           \frac{1}{\lambda(\zeta)}d\zeta\right) d\xi
                         \, .
\end{align}
Let us denote $k_1 = \left(\rho-\tilde{\rho}\right)q$ and $k_2 = k_Iq\left(1+l_1(1)\lambda(1)\right)$. The characteristic equation of~(\ref{eq:NDE_perturbed}) is given by
\begin{equation}
  \label{eq:char-eq}
  s - \left(k_1s +k_2\right)e^{-s \tau} = 0 \, .
\end{equation}
%The following has been proved.
We recall the following theorem that gives conditions to ensure the stability of~$\eqref{eq:NDE_perturbed}$ in the absence of disturbances. 
\begin{theorem}\cite{CT15}
  \label{theo:coron_tamasoiu}
  Let us assume that $k_2 \neq 0$. The characteristic equation~\eqref{eq:char-eq} has its zeroes in the complex half-left plane if and only if the feedback parameters $k_1$ and $k_2$ satisfy $\left|k_1\right|<1$, $k_2 < 0$ and the time delay $\tau$ is such that $\tau \in \left(0,\tau_0\right)$ where $\tau_0$ is defined by
  \begin{align}
    \tau_0 & = 
        -\frac{\sqrt{1-k_1^2}}{\left|k_2\right|}\arctan\left(\frac{\sqrt{1-k_1^2}}{\left|k_1\right|}\right) \nonumber \\
    & \hphantom{=} + \frac{\pi\sqrt{1-k_1^2}}{\left|k_2\right|} \, ,\quad  \text{ if } k_1 \in (-1,0) \\
    \tau_0 & =\frac{\pi}{2\left|k_2\right|}\, , \quad \text{ if } k_1 = 0 \, , \\
    \tau_0 & = \frac{\sqrt{1-k_1^2}}{\left|k_2\right|} \arctan\left(\frac{\sqrt{1-k_1^2}}{k_2}\right)\, , \, \text{if } k_1 \in (0,1) \, .
  \end{align}
\end{theorem}
We recall the definition of Input-to-State Stability (ISS).
\begin{definition}
  The system described by the equations~\eqref{eq:NDE_perturbed} is said to be
  Input-to-State Stable (ISS) if there exist a~$\KL$ function~$f$
  and a~$\K$ function~$g$ such that, for any bounded initial
state~$\left(\overline{\alpha}^0,\overline{\beta}^0\right)^\top$ and any measurable locally essentially bounded
input~$K$, the solution exists for all~$t\geq 0$, and furthermore it
satisfies
\begin{align}
  \left|
    \overline{\alpha}(t,1)
  \right| & \leq
f\left(\left\lVert\overline{\alpha}_0\right\rVert_\infty+\left\lVert
    \overline{\beta}_0\right\rVert_\infty,t\right) \nonumber \\
  & \hphantom{\leq} + g\left(\left\lVert
    K\right\rVert_{L^\infty((0,t);\R)}\right) \, .
\end{align}
\end{definition}
Using this result and the fact that $\left(\rho-\tilde{\rho}\right)q<1$ we may state the following Proposition assessing the ISS of system~\eqref{eq:NDE_perturbed}.
\begin{proposition}
  \label{prop:ISS_NDE}
	Let us choose $k_I$ such that conditions of Theorem~\ref{theo:coron_tamasoiu} for $k_1 = \left(\rho-\tilde{\rho}\right)q$ and $k_2= k_Iq\left(1+l_1(1)\lambda(1)\right)$ hold, then system~\eqref{eq:NDE_perturbed} is ISS with
  respect to the input~$K$.
\end{proposition}
\begin{proof}
  Let us denote $z(t) = \overline{\alpha}(t,1)$. The variation-of-constants formula for the NDE~\eqref{eq:NDE_perturbed} reads (see~\cite{HL13} page~31)
  \begin{align}
    z\left(\left(\overline{\alpha}^0,\overline{\beta}^0\right),K\right)(t) & = z\left(\left(\overline{\alpha^0},\overline{\beta}^0\right),0\right)(t) \nonumber \\
    & \hphantom{=} + \int_0^t X(t-s)K(s)ds \, , \label{eq:representation_formula}
  \end{align}
  where $z\left(\left(\overline{\alpha}^0,\overline{\beta}^0\right),0\right)(t)$ denotes the solution of the homogeneous NDE~\eqref{eq:NDE_perturbed} (i.e. when $K \equiv 0$) in term of the fundamental solution $X$ (see~\cite{HL13} for a definition of the fundamental solution). Theorem~7.6 page 32 in~\cite{HL13} guarantees that if $s_0$ is the supremum of the real part of the roots of the characteristic equation~\eqref{eq:char-eq} then for any $s > s_0$ there exists $k = k\left(s\right)$ such that  the fundamental solution $X$ satisfies the inequality
  \begin{equation}
    \label{eq:inequality_X}
    \left\lVert X(t) \right\Vert \leq ke^{s t} \, , \quad t \geq 0 \, .
  \end{equation}
  Conditions of Theorem~\ref{theo:coron_tamasoiu} ensure that $s_0 <
  0$ and consequently that there exists $s<0$ and $k$ such
  that inequality~\eqref{eq:inequality_X} holds. Then, using this bound together with the representation formula~\eqref{eq:representation_formula} we immediately conclude the proof of Proposition~\ref{prop:ISS_NDE}.
\end{proof}

\subsection{Output Regulation}
\label{sec:output-regulation-1}

The following theorem assesses the output regulation of
system~(\ref{eq:perturbed_system_1})--(\ref{eq:perturbed_system_4}),~(\ref{eq:Controller}),~\eqref{eq:eta_dot},
and~\eqref{eq:U_2}.
\begin{theorem}
  \label{theo:rejection_1}
  Consider system~(\ref{eq:perturbed_system_1}),~(\ref{eq:perturbed_system_2}) with
  boundary
  conditions~(\ref{eq:perturbed_system_3}),~(\ref{eq:perturbed_system_4})
  where $U$ is given by~\eqref{eq:Controller} with $U_{BS}$ given by~\eqref{eq:U_2}, $\eta$
  satisfying~\eqref{eq:eta_dot}, and with bounded initial conditions
  $\left(u^0,v^0,\eta^0\right) \in E$. Then, assuming that conditions of Proposition~\ref{prop:ISS_NDE} hold, there exists a positive constant $M$ such
  that the controlled output $y(t)$ satisfies
  \begin{equation}
    \label{eq:boundedness_y}
    \left| y(t)\right| \leq M \, .
  \end{equation}
  Furthermore, if $\partial_t d_1 = \partial_t d_2 = \dot{d}_3 =
  \dot{d}_4 = n = 0$, then the controlled output satisfies
  \begin{align}
    \label{eq:rejection_1}
    \lim_{t\rightarrow \infty} \left| y(t)\right| = 0\, .
  \end{align}
\end{theorem}
\begin{proof}
  Let us recall that one has
  \begin{align}
    \lim_{t\rightarrow \infty} |u(t,1)|& = \lim_{t\rightarrow \infty} \left|\alpha(t,1)+\int_0^1L^{\alpha \alpha}(1,\xi) \alpha(t,\xi)d\xi \right. \nonumber \\
    & \hphantom{=} \left.+ \int_0^1 L^{\alpha \beta}(1,\xi) \beta(t,\xi) d\xi\right| \nonumber \\
                 & =\lim_{t\rightarrow \infty}
                   \left|\bar{\alpha}(t,1)+\int_0^1L^{\alpha
                   \alpha}(1,\xi) \bar{\alpha}(t,\xi)  d\xi \right.\nonumber \\
    & \hphantom{=} \left. + \int_0^1L^{\alpha\beta}(1,\xi) \bar{\beta}(t,\xi)
                   d\xi\right| \, . \label{eq:lim_u1}
  \end{align}
  Now let us observe that for all $t \geq \tau$ and all $\theta \in [0,x]$,
  \begin{align}
     \label{eq:alpha_bar_relation_proof_theo_1} \overline{\alpha}(t,x) & = \overline{\alpha}\left(t-\int_\theta^x
    \frac{1}{\lambda(\zeta)}d\zeta,\theta\right)     \nonumber \\
    & \hphantom{=} - \int_\theta^x
    \frac{1}{\lambda(\zeta)}\alpha_t^{ss} \left(t-\int_\zeta^x
                             \frac{1}{\lambda(s)}ds,s\right)d\zeta \, ,
    \\
    \overline{\beta}(t,x) % & = \overline{\beta}\left(t+\int_0^x
    % \frac{1}{\mu(\xi)}d\xi,0\right) \nonumber \\
    % & \hphantom{=} + \int_0^x
    % \frac{1}{\mu(\xi)}\beta^{ss}_t\left(t+\int_\xi^x
    %                         \frac{1}{\mu(\zeta)}d\zeta,\xi\right)d\xi
    %                         \nonumber \\
    & =
      \frac{1}{q}\overline{\alpha}\left(t+\phi_2(x),0\right) \nonumber \\
    & \hphantom{=} + \int_0^x
    \frac{1}{\mu(\xi)}\beta^{ss}_t\left(t+\int_\xi^x
                            \frac{1}{\mu(\zeta)}d\zeta,\xi\right)d\xi
      \, . \label{eq:beta_bar_relation_proof_theo_1}
  \end{align}
  Besides, Lemma~\ref{lem:existence_ss} ensures that
  $\alpha_t^{ss}$ and $\beta_t^{ss}$ are bounded. Therefore, relationships~\eqref{eq:alpha_bar_relation_proof_theo_1}
  and~\eqref{eq:beta_bar_relation_proof_theo_1} combined with the ISS
  of $\overline{\alpha}(t,1)$ as proved in
  Proposition~\ref{prop:ISS_NDE} ensure that
  $\overline{\alpha}(t,x)$ and $\overline{\beta}(t,x)$ are bounded for
  all $x \in [0,1]$. Then, with~\eqref{eq:lim_u1} one
  gets~\eqref{eq:boundedness_y}. Now, if $\partial_t d_1 = \partial_t d_2 = \dot{d}_3 =
  \dot{d}_4 = n = 0$, then $\alpha_t^{ss} = \beta_t^{ss} =
  \dot{\eta}^{ss} = 0$ where $\alpha^{ss}$ and $\beta^{ss}$ are
  solutions to the ODE given in~\eqref{steady_eq}, and
  $\eta^{ss}$ is given in~\eqref{eq:eta_ss}. In virtue of the ISS of system~\eqref{eq:NDE_perturbed} stated in Proposition~\ref{prop:ISS_NDE} and using the relationships~\eqref{eq:alpha_bar_relation_proof_theo_1}
  and~\eqref{eq:beta_bar_relation_proof_theo_1} one has
  \begin{align}
    \lim_{t\rightarrow \infty} |u(t,1)|% & = \lim_{t\rightarrow \infty}
                                       %   \left|\bar{\alpha}(t,1)+\int_0^1L^{\alpha
                                       %   \alpha}(1,\xi) \bar{\alpha}(t,\xi)  d\xi \right.\nonumber \\
                                       % & \hphantom{=} \left. + \int_0^1L^{\alpha\beta}(1,\xi) \bar{\beta}(t,\xi)
                                       %   d\xi\right| \nonumber \\
                                       & =  \left|\alpha^{ss}(1)+\int_0^1L^{\alpha \alpha}(1,\xi) \alpha^{ss}(\xi)d\xi \right. \nonumber \\
                                       & \hphantom{=} \left.+ \int_0^1 L^{\alpha \beta}(1,\xi) \beta^{ss}(\xi) d\xi\right| = 0 \, .
  \end{align}
  This concludes the proof of Theorem~\ref{theo:rejection_1}.
\end{proof}

%%% Local Variables:
%%% mode: latex
%%% TeX-master: "ACC_I_Theoretical"
%%% End:

\section{Boundary Observer}
\label{sec:boundary-observer}

In this section we design an observer that relies on the noisy measurements at the right boundary:~${y_m(t)=u(t,1)+n(t)}$. This observer will be designed as a function of a parameter~$\epsilon$ that can be interpreted as a measure of trust in our measurements relative to the model (or unmeasured disturbances).

\subsection{Observer Design}

Similarly to \cite{VKC11}, the observer equations are set as follows
\begin{align}
  \hat{u}_t+\lambda(x)\hat{u}_x=&\gamma_1(x)\hat{v}
                                  -P^+(x)\left(\hat{u}(t,1)-y_m(t)\right) \label{hat_u} \\
  \hat{v}_t-\mu(x) \hat{v}_x=&\gamma_2(x)\hat{u}
                               -P^-(x)\left(\hat{u}(t,1)-y_m(t)\right)
                               \, , \label{hat_v}
\end{align}
with the modified boundary conditions 
\begin{align}
  &\hat{u}(t,0)=q\hat{v}(t,0) \label{hat_boundary1}\\
  &\hat{v}(t,1)=\rho (1-\epsilon) \hat{u}(t,1) +\rho \epsilon y_m(t) +
    U(t) \, . \label{hat_boundary}
\end{align}
The gains~$P^+(\cdot)$ and~$P^-(\cdot)$ are defined as
\begin{align}
  P^+(x) & =-\lambda(x) P^{uu}(x,1)+\mu(x) \rho (1-\epsilon) P^{uv}(x,1) \label{eq_P_eps_1} \\
  P^-(x) & =-\lambda(x) P^{vu}(x,1)+\mu(x) \rho (1-\epsilon) P^{vv}(x,1)
           \, , \label{eq_P_eps_2}
\end{align} 
where the kernels~$P^{uu}, P^{uv}, P^{vu}$, and~$P^{vv}$ are defined in~\cite{VKC11}.
\begin{remark}
  The coefficient~$\epsilon \in [0,1]$ in~\eqref{hat_boundary} can be
  interpreted as a measure of trust in our measurements relative to
  the model (or unmeasured disturbances), where~$\epsilon=1$ results
  in relying more on the measurements and~$\epsilon=0$ relying more on
  the model. This trade-off will be made explicit in terms of the magnitude of
 ~$d_i$,~$i=1,\dots,4$ relative to~$n$ in the following.
\end{remark}

\begin{remark}
  The coefficient~$\epsilon$ cannot be chosen arbitrarily in~$[0,1]$. As it will appear in the next subsection, it has to be close enough to 1 to ensure the convergence of the observer.
\end{remark}

Combining the observer \eqref{hat_u}--\eqref{hat_boundary} to the system \eqref{eq:perturbed_system_1}--\eqref{eq:perturbed_system_4} yields the error system (denoting~$\tilde{u}(t,x)=u(t,x)-\hat{u}(t,x)$ and~$\tilde{v}(t,x)=v(t,x)-\hat{v}(t,x)$):
\begin{align}
  \tilde{u}_t+\lambda(x) \tilde{u}_x& =\gamma_1(x)\tilde{v} -P^+(x)\tilde{u}(t,1)\nonumber \\
                                              &\hphantom{=}
                                                -n(t)P^+(x) +d_1(t)m_1(x) \label{tilde_u}\\
  \tilde{v}_t-\mu(x) \tilde{v}_x& =\gamma_2(x)\tilde{u} -P^-(x)\tilde{u}(t,1)\nonumber \\
                                              & \hphantom{=}
                                                -n(t)P^-(x)  +
                                                d_2(t)m_2(x) \, , \label{tilde_v}
\end{align}
with the boundary conditions 
\begin{align}
  \tilde{u}(t,0)&=q\tilde{v}(t,0)+d_3(t),\label{tilde_boundary0} \\
 \tilde{v}(t,1)&=\rho(1-\epsilon)\tilde{u}(t,1)+d_4(t)-\rho \epsilon
    n(t) \, .\label{tilde_boundary}
\end{align}

\subsection{Ideal Error System}
In this section, we consider the unperturbed system with uncorrupted measurements; to give insight on the impact of~$\epsilon$ in the ideal case. Using the backstepping approach and a Volterra transformation identical to the one presented in~\cite{VKC11}, we can map system \eqref{tilde_u}--\eqref{tilde_boundary} to a simpler target system.
Consider the kernels~$P^{uu}, P^{uv}, P^{vu}$, and~$P^{vv}$ defined in~\cite{VKC11} and the following Volterra transformation
\begin{align}
  \tilde{u}(t,x)=\tilde{\alpha}_{id}(t,x)-\int_x^1(P^{uu}(x,\xi)\tilde{\alpha}_{id}(t,\xi) \nonumber \\
  +P^{uv}(x,\xi)\tilde{\beta}_{id}(t,\xi))d\xi  \label{eq_M0} \\
  \tilde{v}(t,x)=\tilde{\beta}_{id}(t,x)-\int_x^1(P^{vu}(x,\xi)\tilde{\alpha}_{id}(t,\xi) \nonumber \\
  +P^{vv}(x,\xi)\tilde{\beta}_{id}(t,\xi))d\xi \, . \label{eq_N0}
\end{align}
Differentiating \eqref{eq_M0} and \eqref{eq_N0} with respect to space and time, one can prove that system \eqref{tilde_u}--\eqref{tilde_boundary} is equivalent to the following system 
\begin{align}
  &(\tilde{\alpha}_{id})_t+\lambda(x) (\tilde{\alpha}_{id})_x=0 \label{alpha_eq}\\
  &(\tilde{\beta}_{id})_t-\mu(x) (\tilde{\beta}_{id})_x=0 \, ,
\end{align} 
with the following boundary conditions
\begin{align}
  \tilde{\alpha}_{id}(t,0)&=q\tilde{\beta}_{id}(t,0)\\
  \tilde{\beta}_{id}(t,1)&=\rho(1-\epsilon)\tilde{\alpha}_{id}(t,1) \, . \label{alpha_bound}
\end{align}
We then have the following lemma (see e.g \cite{A17} for details).
\begin{lemma}
  System \eqref{alpha_eq}--\eqref{alpha_bound} is exponentially stable if and only if 
  \begin{align}
    1-\frac{1}{|\rho q|} < \epsilon\leq 1 \, . \label{epsilon}
  \end{align}
\end{lemma}

\begin{remark}
  In the case~$\epsilon=1$  we have the same target system as the one presented in~\cite{VKC11}. It converges in finite time~$\tau$ to zero.
\end{remark}
Note that due to Assumption~\ref{assum:function_regularity} the proposed interval is non-empty.

% \begin{lemma}
%   Consider system \eqref{eq_P_uu}-\eqref{eq_P_vv}. There exists a unique solution~$P^{uu}$,~$P^{uv}$,~$P^{vu}$ and~$P^{vv}$ in~$L^{\infty}(\mathcal{T})$. 
% \end{lemma}
% \begin{proof}
%   The proof of this theorem is the same as the one given in \cite[Theorem 4]{VKC11}. Classically (see \cite{hu2015control}, \cite{john1960continuous} and \cite{whitham2011linear}) it consists in transforming the kernel equations into integral equations using the method of characteristics. These integral equations are then solved recursively using the method of successive approximations. The invertibility of the transformation is a property of Volterra transformations.
% \end{proof}

\subsection{Error System including Noise and Disturbance}
We consider in this section the real error-system~\mbox{\eqref{tilde_u}--\eqref{tilde_boundary}}, including the noise and
disturbances~$n$,~$d_i$,~$i=1,\dots,4$. Applying the Volterra transformations~\eqref{eq_M0} and~\eqref{eq_N0}, system~\eqref{tilde_u}--\eqref{tilde_boundary} is mapped to the following target system 
\begin{align}
  \tilde{\alpha}_t+\lambda(x)
    \tilde{\alpha}_x & =n(t)f_1(x)+d_1(t)f_2(x) \nonumber \\
  & \hphantom{=} + d_2(t)f_3(x)+d_4(t)f_4(x) \label{alpha_eq_disturbed}\\
  \tilde{\beta}_t-\mu(x)
    \tilde{\beta}_x & = n(t)g_1(x)+d_1(t)g_2(x) \nonumber \\
  & \hphantom{=} +d_2(t)g_3(x)+d_4(t)g_4(x) \, , \label{beta_eq_disturbed}
\end{align} 
with the boundary conditions
\begin{align}
  \tilde{\alpha}(t,0)&=q\tilde{\beta}(t,0)+d_3(t) \label{alpha_bound_disturbed} \\
  \tilde{\beta}(t,1)&=\rho(1-\epsilon)\tilde{\alpha}(t,1)+d_4(t)-\rho\epsilon
                      n(t) \, , \label{beta_bound_disturbed}
\end{align}
where~$f_i$,~$i=1,\dots,8$, are the solutions of the following integral equations
\begin{align}
  f_1(x) & = \int_x^1\left( P^{uu}(x,\xi)f_1(\xi)+P^{uv}(x,\xi)g_1(\xi)\right)d\xi\nonumber \\
         & \hphantom{=}  -P^+(x) -\mu(1)\rho\epsilon
           P^{uv}(x,1) \label{eq:f_1} \\
  f_2(x) & = m_1(x) + \int_x^1 P^{uu}(x,\xi)f_2(\xi)d\xi \nonumber \\
         & \hphantom{=} + \int_x^1P^{uv}(x,\xi)g_2(\xi)d\xi \\
  f_3(x) & = \int_x^1
           \left(P^{uu}(x,\xi)f_3(\xi)+P^{uv}(x,\xi)g_3(\xi)\right)d\xi\\
  f_4(x) & = \mu(1)P^{uv}(x,1) + \int_x^1 P^{uu}(x,\xi)f_4(\xi)d\xi
           \nonumber \\
  & \hphantom{=} + \int_x^1 P^{uv}(x,\xi)g_4(\xi)d\xi \\
  g_1(x) & =\int_x^1\left(P^{vu}(x,\xi)f_1(\xi) + P^{vv}(x,\xi)g_1(\xi)\right)d\xi \nonumber \\
  & \hphantom{=} -P^-(x) -\mu(1)\rho\epsilon P^{vv}(x,1) \\
  g_2(x) & = \int_x^1
           \left(P^{uv}(x,\xi)f_2(\xi)+P^{vv}(x,\xi)g_2(\xi)\right)d\xi
  \\
  g_3(x) & = m_2(x) + \int_x^1 P^{vu}(x,\xi)f_3(\xi)d\xi \nonumber \\
  & \hphantom{=} +\int_x^1 P^{vv}(x,\xi)g_3(\xi)d\xi \\
  g_4(x) & = \mu(1)P^{vv}(x,1) + \int_x^1 P^{vu}(x,\xi)f_4(\xi)d\xi
           \nonumber \\
  & \hphantom{=} + \int_x^1 P^{vv}(x,\xi)g_4(\xi)d\xi \label{eq:g_4} \, .
\end{align}
The functions $f_i$ and $g_i$ are well defined as solution of an integral equation \cite{Y60}.
The following theorem states that the system is ISS with
respect to~$n$ and~$d_i$,~$i=1,\dots,4$, and thus remains
stable in presence of bounded noise and disturbances

\begin{proposition}
  \label{prop:iss_observer}Let us
  assume that~$\rho$,~$q$, and~$\epsilon$ satisfy~\eqref{epsilon}. Then, system~\eqref{alpha_eq_disturbed},~\eqref{beta_eq_disturbed} with boundary
  conditions~\eqref{alpha_bound_disturbed} and~\eqref{beta_bound_disturbed} is ISS
  with respect to~$n$ and~$d_i$,~$i=1,\dots,4$. More precisely there
  exist a~$\KL$ function~$h_1$ and a~$\K$ function~$h_2$ such that for
  any initial condition~$\left(\tilde{\alpha}^0,\tilde{\beta}^0\right)^\top \in E'$ the following
  holds, for all $t\geq 0$,
  \begin{align}
    \left\lVert \left(\tilde{\alpha},\tilde{\beta}\right)^\top
    \right\rVert_{E'} & \leq h_2\left(\left\lVert \left(n,d_1,\dots,d_4\right)^\top
    \right\rVert_{L^\infty\left((0,t);\R^5\right)}\right) \nonumber \\
    & \hphantom{=} +   h_1\left(\left(\tilde{\alpha}^0,\tilde{\beta}^0\right)^\top,t\right) \, . \label{eq:ISS_tilde_alpha_tilde_beta}
  \end{align}
\end{proposition}

\begin{proof}
  The mechanisms of the proof use the characteristics
  method and an iteration process. For the sake of simplicity we introduce the notations~$\underline{\lambda}$, $\underline{\mu}$, $K_1$,~$K_2$, and~$\tilde{d}$
  \begin{align}
    \underline{\lambda} & = \min_{x\in [0,1]} \lambda(x),\quad \underline{\mu} = \min_{x \in [0,1]} \mu(x) \\
  K_1(t,x) & = n(t)f_1(x)+d_1(t)f_2(x) \nonumber \\
    & \hphantom{=} + d_2(t)f_3(x)+d_4(t)f_4(x) \label{eq:K_1}
      \end{align}
  \begin{align}
    K_2(t,x) & = n(t)g_1(x)+d_1(t)g_2(x) \nonumber \\
    & \hphantom{=} + d_2(t)g_3(x)+d_4(t)g_4(x) \label{eq:K_2} \\
    \tilde{d}(t) & = d_4(t)-\rho\epsilon n(t) \, . \label{eq:tilde_d}
  \end{align}
  In what follows, for the sake of brevity we write
 ~$\left|K_{1_{[0,t)}}\right|_\infty$ for
 ~$\left|K_1\right|_{L^\infty([0,t)\times(0,1))}$. By the characteristics method we have
  \begin{align}
    \tilde{\alpha}&(\tau,x) % & = \tilde{\alpha}(\tau-\phi_1(x),0) \nonumber \\
    % & \hphantom{=} + \int_0^x \frac{1}{\lambda(\xi)}K_1\left(t-\int_\xi^x \frac{1}{\lambda(\zeta)}d\zeta,\xi\right)d\xi \nonumber \\
    % & = q\tilde{\beta}\left(\tau - \phi_1(x),0\right) + d_3(\tau-\phi_1(x))  \nonumber \\
    % & \hphantom{=} + \int_0^x \frac{1}{\lambda(\xi)}K_1\left(t-\int_\xi^x \frac{1}{\lambda(\zeta)}d\zeta,\xi\right)d\xi \nonumber \\
    % & = q\tilde{\beta}\left(\tau - \phi_1(x) - \phi_2(1), 1\right) + d_3(\tau-\phi_1(x)) \nonumber \\
    % & \hphantom{=} + \int_0^x \frac{1}{\lambda(\xi)}K_1\left(t-\int_\xi^x \frac{1}{\lambda(\zeta)}d\zeta,\xi\right)d\xi \nonumber \\
    % & \hphantom{=} + \int_0^1 \frac{1}{\mu(\xi)}K_2\left(\tau - \phi_1(x)-\int_0^\xi\frac{1}{\mu(\zeta)}d\zeta,\xi\right)d\xi \nonumber \\
    % & = q\rho(1-\epsilon)\tilde{\alpha}(\tau - \phi_1(x) - \phi_2(1),1) \nonumber \\
    % & \hphantom{=} + d_3(\tau-\phi_1(x)) + \tilde{d}(\tau - \phi_1(x) - \phi_2(1))\nonumber \\
    % & \hphantom{=} + \int_0^x \frac{1}{\lambda(\xi)}K_1\left(t-\int_\xi^x \frac{1}{\lambda(\zeta)}d\zeta,\xi\right)d\xi \nonumber \\
    % & \hphantom{=} + \int_0^1 \frac{1}{\mu(\xi)}K_2\left(\tau - \phi_1(x)-\int_0^\xi\frac{1}{\mu(\zeta)}d\zeta,\xi\right)d\xi \nonumber \\
    =  d_3(\tau-\phi_1(x)) + \tilde{d}\left(\tau - \phi_1(x) - \tau_2\right) \nonumber \\
    & \hphantom{=} + q\rho(1-\epsilon) \left(\vphantom{\int_0^{\tau - \phi_1(x) - \phi_2(1)}}\tilde{\alpha}^0(x)+\int_0^{\tau_1 - \phi_1(x)}K_1\left(\xi,w(x,\xi)\right)d\xi \right) \nonumber \\
    & \hphantom{=} + \int_0^x \frac{K_1\left(t-\int_\xi^x \frac{1}{\lambda(\zeta)}d\zeta,\xi\right)}{\lambda(\xi)}d\xi  \nonumber \\
                             & \hphantom{=} + \int_0^1 q\frac{K_2\left(\tau - \phi_1(x)-\phi_2(\xi),\xi\right)}{\mu(\xi)}d\xi  \, ,
  \end{align}
  where $w(x,\xi) = \phi_1^{-1}\left(\phi_1(x)+\xi\right)$. Therefore, one has
  \begin{align}
    \left|\tilde{\alpha}\left(\tau,x\right)\right| & \leq \left|q\rho(1-\epsilon)\right| \left|\tilde{\alpha}^0\right|_\infty  + \left| d_{3_{[0,\tau)}} \right|_\infty + \left|\tilde{d}_{[0,\tau)}\right|_\infty \nonumber \\
    & \hphantom{\leq} +\left(\frac{1}{\underline{\lambda}}+ \left|q\rho(1-\epsilon)\right|\tau\right)\left|K_{1_{[0,\tau)}}\right|_\infty \nonumber \\
    & \hphantom{\leq} + \frac{\left|K_{2_{[0,\tau)}}\right|_\infty}{\underline{\mu}}  \, .
  \end{align}
  Recursively, we get
  \begin{align}
    \left|\tilde{\alpha}\left(n\tau,x\right)\right|_\infty & \leq \left|q\rho(1-\epsilon)\right|^n \left|\tilde{\alpha}^0\right|_\infty \nonumber  \\
                                                               & \hphantom{\leq} +\frac{1}{\underline{\lambda}} \sum_{i=1}^n\left|q\rho(1-\epsilon)\right|^{i-1}\left|K_{1_{[0,n\tau)}}\right|_\infty \nonumber \\
                                                               & \hphantom{\leq} + \tau \sum_{i=1}^n\left|q\rho(1-\epsilon)\right|^i\left|K_{1_{[0,n\tau)}}\right|_\infty \nonumber \\
                                                               & \hphantom{\leq} + \frac{1}{\underline{\mu}}  \sum_{i=1}^n\left|q\rho(1-\epsilon)\right|^{i-1}\left|K_{2_{[0,n\tau)}}\right|_\infty \nonumber \\
                                                               & \hphantom{\leq} + \sum_{i=1}^n \left|q\rho(1-\epsilon)\right|^{i-1}\left| d_{3_{[0,n\tau)}} \right|_\infty \nonumber \\
                                                               & \hphantom{\leq} + \sum_{i=1}^n \left|q\rho(1-\epsilon)\right|^{i-1}  \left|\tilde{d}_{[0,n\tau)}\right|_\infty \, .
  \end{align}
  Using the condition~\eqref{epsilon}, one has~$\left|q\rho(1-\epsilon)\right|<1$ it follows
  \begin{align}
    \left|\tilde{\alpha}\left(n\tau,x\right)\right| & \leq \left|q\rho(1-\epsilon)\right|^n \left\lVert\left(\tilde{\alpha}^0,\tilde{\beta}^0\right)^\top\right\rVert_E \nonumber  \\
                                                               & \hphantom{\leq} + \left(\tau+ \frac{1}{\underline{\lambda}}\right)\frac{\left|K_{1_{[0,n\tau)}}\right|_\infty}{1 - \left|q\rho(1-\epsilon)\right|} \nonumber \\
                                                               & \hphantom{\leq} + \frac{\left|K_{2_[0,n\tau)}\right|_\infty}{\underline{\mu}-\underline{\mu}\left|q\rho(1-\epsilon)\right|}  + \frac{\left| d_{3_{[0,n\tau)}} \right|_\infty}{1 - \left|q\rho(1-\epsilon)\right|} \nonumber \\
    & \hphantom{\leq} + \frac{\left|\tilde{d}_{[0,n\tau)}\right|_\infty}{1 - \left|q\rho(1-\epsilon)\right|} \, .
  \end{align}
  The computation
  showed for~$\tilde{\alpha}$ can be done in a similar way for~$\tilde{\beta}$. We get that for all~$t$ and all~$x$ such
  that~${n\tau \leq t - \phi_1(x) < (n+1)\tau}$
  \begin{align}
    \left|\tilde{\alpha}(t,x)\right| & \leq (1+\left|q\right|)\left|q\rho(1-\epsilon)\right|^n \left\lVert\left(\tilde{\alpha}^0,\tilde{\beta}^0\right)^\top\right\rVert_E \nonumber  \\
                                                               & \hphantom{\leq} +(1+\left|q\right|)\left(\tau+\frac{2}{\underline{\lambda}}\right)\frac{\left|K_{1_{[0,t)}}\right|_\infty}{1 - \left|q\rho(1-\epsilon)\right|} \nonumber \\
                                                               & \hphantom{\leq} +(1+\left|q\right|) \left(\tau+\frac{2}{\underline{\mu}}\right)\frac{\left|K_{2_{[0,t)}}\right|_\infty}{1-\left|q\rho(1-\epsilon)\right|}  \nonumber \\
                                                               & \hphantom{\leq} + (1+\left|q\right|)\frac{2\left| d_{3_{[0,t)}} \right|_\infty}{1 - \left|q\rho(1-\epsilon)\right|} + \left|d_{3_{[0,t)}}\right| \nonumber \\
    & \hphantom{=} + (1+\left|q\right|)\frac{2\left|\tilde{d}_{[0,t)}\right|_\infty}{1 - \left|q\rho(1-\epsilon)\right|} + \left|q\right|\tau \left|K_{2_{[0,t)}}\right|_\infty \nonumber \\
                                     & \hphantom{=} + \left(\tau+\frac{1}{\underline{\lambda}}\right) \left|K_{1_{[0,t)}}\right|_\infty   \, .
  \end{align}
  Finally, with the computations for~$\tilde{\beta}$ we prove
  that~\eqref{eq:ISS_tilde_alpha_tilde_beta} holds with
  \begin{align}
    h_1(X,t) & = Ce^{-\nu t}X \\
    h_2(X) & =
             \left(2\frac{(2+\left|q\right|+\left|\rho(1-\epsilon)\right|)}{1
             - \left|q\rho(1-\epsilon)\right|}\left(\tau+
             \frac{1}{\underline{\lambda}}+\frac{1}{\underline{\mu}}+2\right)\right. \nonumber
    \\
    & \hphantom{=} +2+ \left|q\right|\tau +
      \left|\rho(1-\epsilon)\right|\tau \nonumber \\
    & \hphantom{=} \left. +
      \left(2\tau+\frac{1}{\underline{\lambda}}+\frac{1}{\underline{\mu}}\right)\right)X
      \, ,
  \end{align}
  with $C = (2 + \left|q\right|+\left|\rho(1-\epsilon)\right|)$ and $\nu = \frac{1}{\tau}\ln \left(\frac{1}{q\rho(1-\epsilon)}\right)$.  This concludes the proof of Proposition~\ref{prop:iss_observer}.
  % Finally, we arrive at
  % \begin{align}
  %   \left\lVert\left(\tilde{\alpha},\tilde{\beta}\right)\right\rVert & \leq (1+\left|q\right|+ \left|\rho(1-\epsilon)\right|)\left|q\rho(1-\epsilon)\right|^n \left\lVert\left(\tilde{\alpha}^0,\tilde{\beta}^0\right)^\top\right\rVert_E \nonumber  \\
  %                                    & \hphantom{\leq} +(1+\left|q\right|)\left(\tau+\frac{2}{\underline{\lambda}}\right)\frac{\left|K_1\right|_\infty}{1 - \left|q\rho(1-\epsilon)\right|} \nonumber \\
  %                                    & \hphantom{\leq} +(1+\left|q\right|) \left(\tau+\frac{2}{\underline{\mu}}\right)\frac{\left|K_2\right|_\infty}{1-\left|q\rho(1-\epsilon)\right|}  \nonumber \\
  %                                    & \hphantom{\leq} + (1+\left|q\right|)\frac{2\left| d_{3_{[0,t)}} \right|_\infty}{1 - \left|q\rho(1-\epsilon)\right|} \nonumber \\
  %                                    & \hphantom{=} + (1+\left|q\right|)\frac{2\left|\tilde{d}_{[0,t)}\right|_\infty}{1 - \left|q\rho(1-\epsilon)\right|} + \left|d_{3_{[0,t)}}\right| \nonumber \\
  %                                    & \hphantom{=} + \left(\tau+\frac{1}{\underline{\lambda}}\right) \left|K_{1_{[0,t)}}\right|_\infty  + \left|q\right|\tau \left|K_{2_{[0,t)}}\right|_\infty \, .
  % \end{align}
\end{proof}

\begin{theorem}
  \label{theo:iss_observer}Let us
  assume that~$\rho$,~$q$, and~$\epsilon$ satisfy the condition
  in~\eqref{epsilon}. Then, system~\eqref{tilde_u},~\eqref{tilde_v} with boundary
  conditions~\eqref{tilde_boundary0} and~\eqref{tilde_boundary} is ISS
  with respect to~$n$ and~$d_i$,~$i=1,\dots,4$. More precisely there
  exist a~$\KL$ function~$h_1$ and a~$\K$ function~$h_2$ such that for
  any initial condition~$\left(\tilde{u}^0,\tilde{v}^0\right)^\top \in E'$ the following
  holds
  \begin{align}
    \left\lVert \left(\tilde{u},\tilde{v}\right)^\top
    \right\rVert_{E'} & \leq  h_2\left(\left\lVert \left(n,d_1,\dots,d_4\right)^\top
    \right\rVert_{L^\infty\left((0,t);\R^5\right)}\right) \nonumber \\
    & \hphantom{=} +  h_1\left(\left(\tilde{u}^0,\tilde{v}^0\right)^\top,t\right) \, . \label{eq:ISS_tilde_alpha_tilde_beta2}
  \end{align}

  \begin{proof}
    Using the fact that the backstepping transformation~\eqref{eq_M0},~\eqref{eq_N0} is invertible and Proposition~\ref{prop:iss_observer}, Theorem~\ref{theo:iss_observer} is proved.
  \end{proof}
\end{theorem}

%%% Local Variables:
%%% mode: latex
%%% TeX-master: "ACC_I_Theoretical"
%%% End:

\section{Feedback Output Regulation}
\label{sec:feedb-outp-regul}

Consider system~\eqref{eq:perturbed_system_1},~\eqref{eq:perturbed_system_2} with boundary conditions~\eqref{eq:perturbed_system_3} and~\eqref{eq:perturbed_system_4} where $U$ is given by~\eqref{eq:Controller} with $U_{BS}$ given by
\begin{align}
  U_{BS}(t) & = -\tilde{\rho}(1-\epsilon)\hat{u}(t,1)  - \left(\rho-\tilde{\rho}\right)\int_0^1K^{uu}(1,\xi)\hat{u}(t,\xi)d\xi \nonumber \\
       & \hphantom{=} - \left(\rho-\tilde{\rho}\right)\int_0^1K^{uv}(1,\xi)\hat{v}(t,\xi)d\xi-\tilde{\rho}\epsilon y_m(t) \nonumber \\
       & \hphantom{=} + \int_0^1 \left(K^{vu}(1,\xi)\hat{u}(t,\xi)+ K^{vv}(1,\xi)\hat{v}(t,\xi)\right)d\xi \nonumber \\
       & \hphantom{=} -k_I\int_0^1
         l_1(\xi)\Gamma_1[(\hat{u},\hat{v})(t)](\xi)d\xi \nonumber \\
       & \hphantom{=} - k_I \int_0^1 l_2(\xi)\Gamma_2[(\hat{u},\hat{v})(t)](\xi)d\xi  \, , \label{eq:U_output_feedback}
\end{align}
where $\hat{u}$ and $\hat{v}$ are the solution to~\eqref{hat_u}--\eqref{hat_boundary}. The aim of this section is to prove that the
output $y(t)$ of the system is still regulated in the sense of Theorem~\ref{theo:rejection_1} with the control law in~\eqref{eq:U_output_feedback}. We have the second main result of this paper.

\begin{theorem}
  \label{theo:rejection_2}
  Consider system~\eqref{eq:perturbed_system_1},~\eqref{eq:perturbed_system_2} with boundary conditions~\eqref{eq:perturbed_system_3} and~\eqref{eq:perturbed_system_4} where $U$ is given by~\eqref{eq:Controller} with $U_{BS}$ given by~\eqref{eq:U_output_feedback}, $\eta$
  satisfying~\eqref{eq:eta_dot}, and with bounded initial conditions
  $\left(u^0,v^0,\eta^0\right) \in E$. Then, assuming that conditions of Proposition~\ref{prop:ISS_NDE} hold, there exists a positive constant $M$ such that the controlled output $y(t)$ satisfies
  \begin{equation}
    \label{eq:disturbance_bound}
    \left| y(t)\right| \leq M \, .    
  \end{equation}
  Furthermore, if $\partial_t d_1 = \partial_t d_2 = \dot{d}_3 =
  \dot{d}_4 = \dot{d}_5 = 0$, then the output satisfies
  \begin{equation}
    \label{eq:convergence_theo}
    \lim_{t\rightarrow \infty} \left| y(t)\right| = 0\, .
  \end{equation}
\end{theorem}

\begin{proof}
  We have $\hat{u} = \hat{u} - u + u = -\tilde{u} + u$ and ${\hat{v}= \hat{v} - v + v = -\tilde{v} + v}$. Therefore, one has
  \begin{align}
    U_{BS}(t) & = -\tilde{\rho}u(t,1) - \left(\rho-\tilde{\rho}\right)\int_0^1K^{uu}(1,\xi)u(t,\xi)d\xi \nonumber \\
         & \hphantom{=} - \left(\rho-\tilde{\rho}\right)\int_0^1K^{uv}(1,\xi)v(t,\xi)d\xi -\tilde{\rho}\epsilon n(t)\nonumber \\
         & \hphantom{=} + \int_0^1 \left(K^{vu}(1,\xi)u(t,\xi)+ K^{vv}(1,\xi)v(t,\xi)\right)d\xi \nonumber \\
         & \hphantom{=} -k_I\int_0^1 l_1(\xi)\Gamma_1[(u,v)(t)](\xi)d\xi \nonumber \\
         & \hphantom{=} - k_I \int_0^1 l_2(\xi)\Gamma_2[(u,v)(t)](\xi)d\xi  + \D(t) \, . \label{eq:U_output_feedback_2}
  \end{align}
  where $\D(t)$ is given by $U_{BS}$ in~\eqref{eq:U_output_feedback} which $\hat{u}$ and $\hat{v}$ have been replaced by $-\tilde{u}$ and $-\tilde{v}$ respectively. Since, $\tilde{u}$ and $\tilde{v}$ are bounded thanks to Theorem~\ref{theo:iss_observer}, we can consider $\D-\tilde{\rho}\epsilon n$ as a new perturbation in the input and we can apply Theorem~\ref{theo:rejection_1} to conclude that~\eqref{eq:disturbance_bound} holds. Now, if the perturbation vanishes then $\D(t)-\tilde{\rho}\epsilon n(t)$ will vanish in virtue of ISS of the observer system and then again by applying Theorem~\ref{theo:rejection_1} we have~\eqref{eq:convergence_theo}. This concludes the proof of Theorem~\ref{theo:rejection_2}.
\end{proof}

%%% Local Variables:
%%% mode: latex
%%% TeX-master: "ACC_I_Theoretical"
%%% End:

\section{Concluding Remarks}
In this paper we have solved the output feedback regulation problem for a system composed of two linear hyperbolic PDEs with collocated boundary input and output in presence of disturbances and noise in the measurements. This has been done by combining in the control law a backstepping approach with an integral term. By transforming the boundary condition of the resulting target system into a Neutral Differential Equation we have proved that this former system is Input-to-State Stable with respect to disturbances and noise. The proposed controller has finally been combined with a backstepping-based observer to ensure output-feedback stabilization of the output. Both the proposed controller and the observer present some degrees of freedom (necessary to ensure robustness to delays) that enables a trade-off between disturbance and noise sensitivity. The effect of such parameters, in particular regarding the systems sensitivity functions are derived in a companion paper.

%%% Local Variables:
%%% mode: latex
%%% TeX-master: "ACC_I_Theoretical"
%%% End:

\bibliographystyle{plain}
\bibliography{biblio}

\end{document}